\documentclass[12pt]{amsart}

\usepackage{amssymb}
\usepackage{graphicx}
\usepackage{enumerate}
\usepackage{multirow}
\usepackage{amsmath,color}
\usepackage{hyperref}
\usepackage{url}
\usepackage{adjustbox}

\newtheorem{thm}{Theorem}[section]

\newtheorem{lem}[thm]{Lemma}
\newtheorem{prop}[thm]{Proposition}

\newtheorem{rem}[thm]{\bf{Remark}}

\newtheorem{exam}[thm]{Example}
\theoremstyle{definition}

\newcommand{\seq}[1]{\langle #1\rangle}
\makeatletter
\@namedef{subjclassname@2020}{%
  \textup{2020} Mathematics Subject Classification}
\makeatother

\title{Vertex-transitive Neumaier graphs}

\author{Mojtaba Jazaeri}
\address{Department of Mathematics, Shahid Chamran University of Ahvaz, Ahvaz, Iran}
\email{M.Jazaeri@scu.ac.ir, M.Jazaeri@ipm.ir}
\begin{document}

\maketitle

\begin{abstract}
A graph $\Gamma$ is called edge-regular whenever it is regular and for any two adjacent vertices, the number of their common neighbors is independent of the choice of vertices. A clique $C$ in $\Gamma$ is called regular whenever for any vertex out of $C$, the number of its neighbors in $C$ is independent of the vertex. A Neumaier graph is a non-complete edge-regular graph with a regular clique. In this paper, we study vertex-transitive Neumaier graphs. We give a necessary and sufficient condition under which a vertex-transitive Neumaier graph is strongly regular. We also identify Neumaier Cayley graphs with small valency at most $10$ among vertex-transitive Neumaier graphs.
\end{abstract}

\section{\bf Introduction}

Neumaier \cite{Neumaier} studied regular cliques in an edge-regular graph and stated the question that is every edge-regular graph with a regular clique strongly regular? Recently, a non-complete edge-regular graph with a regular clique is called a Neumaier graph and it has attracted a great deal of attention among authors, see \cite{ACDKZ}, \cite{ADDK}, \cite{ADZ}. It is known that a Neumaier graph has diameter $2$ or $3$.  Soicher \cite{Soicher} studied cliques in an edge-regular graph. Greaves and Koolen \cite{GK} answered the Neumaier question by constructing an infinite number of examples of Neumaier graphs that are not strongly regular.  Evans  \cite{Evans} studied regular induced subgraphs in an edge-regular graph. Abiad et al. \cite{ADDK} studied Neumaier graphs with few eigenvalues and proved that there is no Neumaier graph with exactly four distinct eigenvalues. Abiad et al. \cite{ACDKZ} studied Neumaier graphs with at most $64$ vertices that are not strongly regular (among other results). They listed the feasible parameters in \cite[Table~1]{ACDKZ}. Abiad et al. completed the existence and nonexistence of some of those feasible parameters in \cite{ADZ}.

A vertex-transitive Neumaier graph has diameter $2$ (see Theorem \ref{main theorem 1} \eqref{vertex-transitive Neumaier graph diameter} below). Several of the known examples of Neumaier graphs are vertex-transitive, see \cite{Evans}, \cite{EGP}. We obtain a new equitable partition with four parts for a Neumaier graph by adding a new condition between the number of adjacent vertices of two special parts (see Theorem \ref{main theorem 1} \eqref{equitable partition}). Using this, we give a necessary and sufficient condition under which a vertex-transitive Neumaier graph is strongly regular (see Theorem \ref{main theorem 1} \eqref{vertex-transitive Neumaier graph diameter}). We characterize (vertex-transitive) strongly regular Neumaier graphs with small valency at most $10$ (see Proposition \ref{Strongly regular Neumaier Cayley graph up to valency 10}). We also identify Neumaier Cayley graphs among the known vertex-transitive Neumaier graphs with small valency at most $10$ (see Proposition \ref{strictly neumaier cayley graphs}).

\section{\bf Preliminaries}
In this paper, all graphs are connected, undirected and simple, i.e., there are no loops or multiple edges. From now on, let $\Gamma$ denote a graph with vertex set $V$. We denote the complement of $\Gamma$ by $\overline{\Gamma}$. Recall that a subset $ C\subseteq V$ is called a clique in $\Gamma$ whenever the induced subgraph on $C$ is complete. The clique $C$ is called a {\it regular clique with nexus $a$} whenever for any vertex out of $C$, the number of its neighbors in $C$ is $a$.
\begin{lem} (see \cite[p.~13]{ADDK}) \label{eigenvalues}
Let $\Gamma$ be a non-complete $k$-regular graph containing a regular clique of size $c$ with nexus $a$. Then $k$ and $c-a-1$ are two nonnegative distinct eigenvalues of $\Gamma$.
\end{lem}
\begin{proof}
Let $C$ be a regular clique of size $c$ with nexus $a$ in $\Gamma$. Then $C$ gives rise to an equitable partition for $\Gamma$ with the following quotient matrix.
\begin{center}
\begin{equation*}
  \left(
    \begin{array}{cc}
      c-1 & k-c+1 \\
      a & k-a \\
    \end{array}
  \right)
\end{equation*}
\end{center}
It follows that $k$ and $c-a-1$ are two eigenvalues of $\Gamma$ (cf. \cite[Lemma~2.3.1]{BH}). Moreover, $a<c<k+1$ because $\Gamma$ is not a complete graph. This completes the proof.
\end{proof}
\begin{lem} (cf. \cite[Lemma~2.5(i)]{EGP}) \label{Counting}
Let $\Gamma$ be a $k$-regular graph containing a regular clique of size $c$ with nexus $a$. Then $c(k-c+1)=(n-c)a$.
\end{lem}
\begin{proof}
Let $C$ be a regular clique of size $c$ with nexus $a$ in $\Gamma$. Then the result follows by the double counting method for the number of edges between the regular clique $C$ and out of $C$.
\end{proof}
A graph $\Gamma$ with $n$ vertices is called {\it edge-regular} with parameters $(n,k,\lambda)$ whenever it is $k$-regular and for any two adjacent vertices the number of their common neighbors is $\lambda$. A non-complete edge-regular graph with a regular clique is called a {\it Neumaier graph}. A Neumaier graph that is not strongly regular is called a {\it strictly Neumaier graph} (cf. \cite{EGP}). It is known that if $\Gamma$ is a Neumaier graph, then the size and nexus of any regular clique in $\Gamma$ is independent of the choice of the regular clique (cf. \cite[Theorem~1.1]{Neumaier}). A Neumaier graph $\Gamma$ is said to have parameters $(n,k,\lambda;a,c)$ whenever $\Gamma$ is edge-regular with parameters $(n,k,\lambda)$ containing a regular clique of size $c$ with nexus $a$. We note that a Neumaier graph has diameter $2$ or $3$ because it has a regular clique with nexus $a \geq 1$.

A {\it strongly regular graph} with parameters $(n,k,\lambda,\mu)$ is a non-complete connected $k$-regular graph with $n$ vertices such that two vertices have $\lambda$ and $\mu$ common neighbors depending on whether these two vertices are adjacent or non-adjacent, respectively. A strongly regular graph with parameters $(n,k,\lambda,\mu)$ is called nontrivial whenever $0 < \mu < k$ (cf. \cite[\S~1.1]{BCN}). Recall that a strongly regular graph has diameter $2$ and exactly three distinct eigenvalues (cf. \cite[\S~1.3]{BCN}). Let $\Gamma$ be a strongly regular graph with parameters $(n,k,\lambda,\mu)$. Assume that $\Gamma$ is a Neumaier graph with parameters $(n,k,\lambda;a,c)$. Then we say that $\Gamma$ is a strongly regular Neumaier graph with parameters $(n,k,\lambda,\mu;a,c)$.

Let $G$ be a (finite) group and $S$ be an inverse-closed subset of $G$ without the identity element. Then the {\it Cayley graph} $Cay(G,S)$ is a graph whose vertex set is $G$ and two vertices $a,b$ are adjacent whenever $ab^{-1} \in S$. We call the subset $S$ of $G$, the connection set of $Cay(G,S)$. From now on, let $Cay(G,S)$ denote a Cayley graph over the group $G$ with the connection set $S$, and $e$ denote the identity element of $G$. We note that $Cay(G,S)$ is connected if and only if the subgroup generated by $S$ (which is denoted by $\seq{S}$) equals $G$ (cf. \cite[Lemma~3.7.4]{GR}). We call a Neumaier graph that is Cayley a {\it Neumaier Cayley graph}. We denote the dihedral group of order $n$ by $D_{n}$, the symmetric group on $n$ letters by $S_{n}$, the alternating group on $n$ letters by $A_{n}$, and the cyclic additive group modulo $n$ by $\mathbb{Z}_{n}$. We note that a circulant graph with $n$ vertices is a Cayley graph over the cyclic group $\mathbb{Z}_{n}$.

A {\it vertex-transitive} graph is a graph whose automorphism group acts transitively on its vertex set. We denote the automorphism of a graph $\Gamma$ by $Aut(\Gamma)$. Recall that a Cayley graph is vertex-transitive (cf. \cite[\S~3.7]{GR}).
\section{\bf Neumaier graphs}
Recall that a Neumaier graph has diameter $2$ or $3$.
\begin{lem} \label{Complete multipartite graph}
Let $\Gamma$ be a Neumaier graph with parameters $(n,k,\lambda;a,c)$. If $\lambda=0$, then $c=2$, $a=1$, and $\Gamma$ is the complete bipartite graph $K_{k,k}$. Moreover, if $c=2$, then $a=1$, $\lambda=0$, and $\Gamma$ is the complete bipartite graph $K_{k,k}$.
\end{lem}
\begin{proof}
By the definition of a Neumaier graph.
\end{proof}
\begin{lem} \label{at most vertices}
Let $\Gamma$ be a Neumaier graph with diameter $2$. Then $\Gamma$ has at most $\max \{1+k+k(k-2),2k\}$ vertices.
\end{lem}
\begin{proof}
If $\lambda=0$, then $\Gamma$ has $2k$ vertices by Lemma \ref{Complete multipartite graph}. If $\lambda>1$, then $\Gamma$ has at most $1+k+k(k-2)$ vertices by \cite[Proposition~1.4.1]{BCN}. This completes the proof.
\end{proof}
\begin{thm} \label{main theorem 1}
Let $\Gamma$ be a Neumaier graph with vertex set $V$ and parameters $(n,k,\lambda;a,c)$. Assume that $C \subset V$ is a regular clique with nexus $a$ and $e \in C$ is an arbitrary vertex. Let $S$ denote the set of neighbors of $e$ in $\Gamma$. Then the following items hold.
\begin{enumerate}[$(i)$]
\item The induced subgraph on $S \setminus C$ is regular. Moreover, if the number of neighbors of any vertex in $V \setminus (S\cup\{e\})$ to the vertices in $S \setminus C$ is independent of the choice of vertex, then $\{\{e\},C \setminus \{e\},S \setminus C,V \setminus (S\cup\{e\})\}$ is an equitable partition for $\Gamma$. \label{equitable partition}
\item Let $\Gamma$ be strongly regular. Then the number of neighbors of any vertex in $V \setminus (S\cup\{e\})$ to the vertices in $S \setminus C$ is independent of the vertex. \label{Strongly regular condition}
\item Let $\Gamma$ be vertex-transitive. Then $\Gamma$ has diameter $2$. Moreover, $\Gamma$ is strongly regular if and only if the number of neighbors of any vertex in $V \setminus (S\cup\{e\})$ to the vertices in $S \setminus C$ is independent of the vertex.
    \label{vertex-transitive Neumaier graph diameter}
\end{enumerate}
\end{thm}
\begin{proof}
\eqref{equitable partition}: We note that $C \setminus \{e\} \subset S$. This implies that every vertex in $C \setminus \{e\}$ has $c-2$ neighbors in $C \setminus \{e\}$, $\lambda-c+2$ neighbors in $S \setminus C$ and therefore $k-\lambda-1$ neighbors in $V \setminus (S\cup\{e\})$. Moreover, every vertex in $S \setminus C$  has $a-1$ neighbors in $C \setminus \{e\}$, $\lambda-a+1$ neighbors in $S \setminus C$, and therefore $k-\lambda-1$ neighbors in $V \setminus (S\cup\{e\})$. Furthermore, every vertex in $V \setminus (S\cup\{e\})$ has the same number $a$ neighbors in $C$ because $C$ is a regular clique with nexus $a$. Therefore the result follows by the hypothesis and noting that $\Gamma$ is $k$-regular (see Figure \ref{Fig1} below).

\eqref{Strongly regular condition}: By \eqref{equitable partition}, noting that $e$ is non-adjacent to any vertex in $V \setminus (S\cup\{e\})$, and the definition of a strongly regular graph.

\eqref{vertex-transitive Neumaier graph diameter}: The distance between $e$ and an arbitrary vertex in $V$ is at most $2$ by \eqref{equitable partition} and noting that $a>0$. Let $x,y \in V$. Then there is an automorphism $\sigma \in Aut(\Gamma)$ such that $\sigma(x)=e$ since $\Gamma$ is vertex-transitive. Let $\sigma(y)=y'$. Then the distance between $x,y$ is the same as $e,y'$. Assume that the distance between $x,y$ is $2$. Then $x,y$ have the same number of common neighbors as $e,y'$. Therefore the result follows by \eqref{equitable partition}, \eqref{Strongly regular condition}.
\end{proof}
\begin{prop} \label{The number of edges}
Let $\Gamma$ be a Neumaier graph with parameters \, \, \, $(n,k,\lambda;a,c)$. Then $\Gamma$ has at least the following number of edges:
\begin{equation*}
k(k-\lambda)+(k-c+1)(a-1)+\frac{(k-c+1)(\lambda-a+1)}{2}+\frac{(c-1)(c-2)}{2}.
\end{equation*}
\end{prop}
\begin{proof}
We count the number of edges of $\Gamma$ by Theorem \ref{main theorem 1}(\ref{equitable partition}). The number of edges from $e$ to $S$ and $S$ to $V \setminus (S\cup\{e\})$ is $k$ and $k(k-\lambda -1)$, respectively. Moreover, the number of edges in $S$ equals
\begin{equation*}
(k-c+1)(a-1)+\frac{(k-c+1)(\lambda-a+1)}{2}+\frac{(c-1)(c-2)}{2}.
\end{equation*}
This completes the proof.
\end{proof}
\begin{center}
\begin{figure}[h]
 \centering
 \includegraphics*[width=1.5\linewidth]{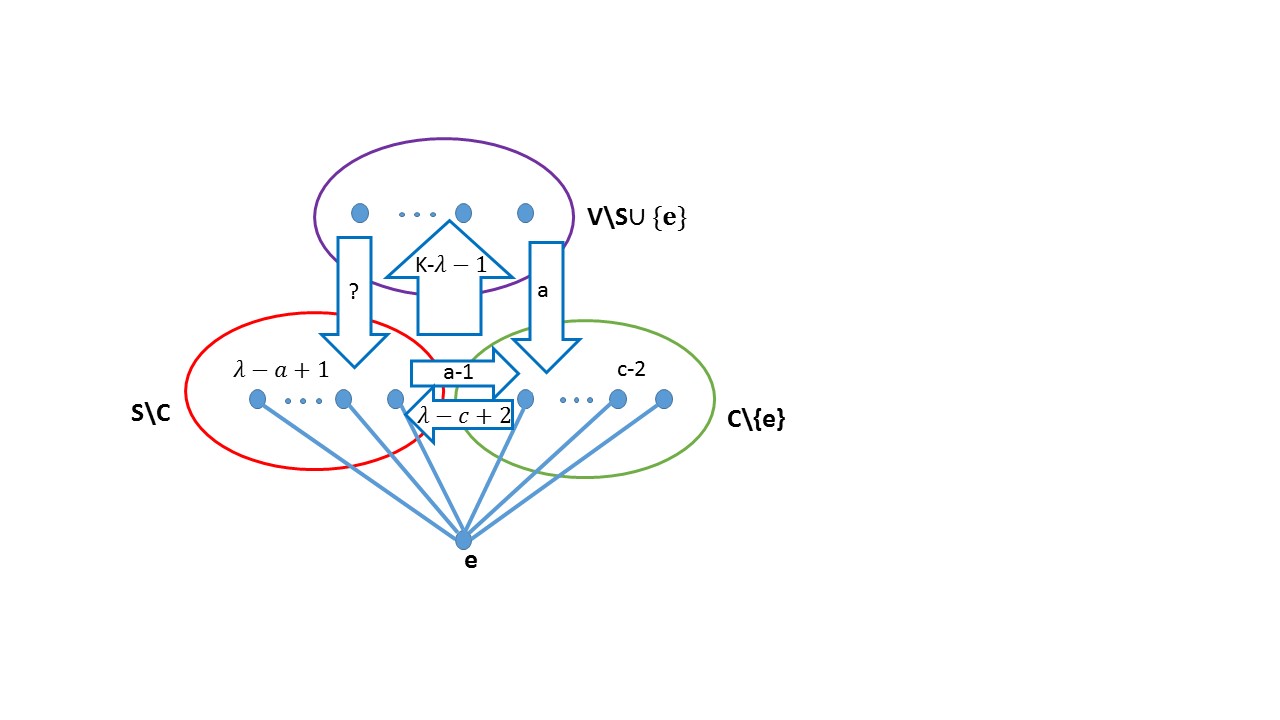}
 \label{diagram}
 \caption{A Neumaier graph with parameters $(n,k,\lambda;a,c)$. Let $C \subset V$ denote a regular clique with nexus $a$ containing an arbitrary element $e$ and $S$ denote the set of neighbors of $e$. The notation on the arrow between two partitions is the number of adjacent vertices from one vertex in a partition to the vertices of another partition. The notation in a partition is the valency of the regular induced subgraph on that partition.} \label{Fig1}
\end{figure}
\end{center}
\begin{lem} (cf. \cite[Lemma~2.5(ii)]{EGP}) \label{Double counting 1}
Let $\Gamma$ be a Neumaier graph with parameters $(n,k,\lambda;a,c)$. Then $(k-c+1)(a-1)=(c-1)(\lambda - c +2)$.
\end{lem}
\begin{proof}
By the double counting method for the number of edges between $S \setminus C$ and $C \setminus \{e\}$ in Theorem \ref{main theorem 1}(\ref{equitable partition}).
\end{proof}
\begin{lem} \label{Double counting 2}
Let $\Gamma$ be a Neumaier graph with parameters $(n,k,\lambda;a,c)$. Then $(c-1)(k-\lambda-1)=(n-k-1)a$.
\end{lem}
\begin{proof}
By the double counting method for the number of edges between $C \setminus \{e\}$ and $V \setminus (S\cup\{e\})$ in Theorem \ref{main theorem 1}(\ref{equitable partition}).
\end{proof}
Recall that a strongly regular graph has exactly three distinct eigenvalues. We have the following result about the eigenvalues of a strongly regular Neumaier graph.
\begin{prop} \label{Integer eigenvalues}
Let $\Gamma$ be a strongly regular Neumaier graph with parameters $(n,k,\lambda,\mu;a,c)$. Then $\Gamma$ has exactly three distinct eigenvalues $k$, $c-a-1$, and $-\frac{\mu}{a}$. Moreover, $-\frac{\mu}{a}$ is an integer.
\end{prop}
\begin{proof}
Let $-m$ be the smallest eigenvalue of $\Gamma$. Then the size of every regular clique admits the Hoffman bound, namely $c=1+\frac{k}{m}$, and $a=\frac{\mu}{m}$ by \cite[Proposition~1.3.2(ii)]{BCN}. This implies that $k$, $c-a-1$, and $-m=-\frac{\mu}{a}$ are all three distinct eigenvalues of $\Gamma$ by Lemma \ref{eigenvalues}. Moreover, $-\frac{\mu}{a}$ is a rational and algebraic integer. This completes the proof.
\end{proof}
\begin{prop} \label{Strongly regular Neumaier graph counting}
Let $\Gamma$ be a strongly regular Neumaier graph with parameters $(n,k,\lambda,\mu;a,c)$. Then
\begin{equation} \label{Strongly regular Neumaier graph Equation 1}
(k-c+1)(k-\lambda-1)=(n-k-1)(\mu-a),
\end{equation}
\begin{equation} \label{Strongly regular Neumaier graph Equation 2}
\mu(c-a-1)=a(k-\mu).
\end{equation}
\end{prop}
\begin{proof}
We have \eqref{Strongly regular Neumaier graph Equation 1} by Theorem \ref{main theorem 1}(\ref{Strongly regular condition}) and the double counting method for the number of edges between $S \setminus C$ and $V \setminus (S\cup\{e\})$ in Theorem \ref{main theorem 1}(\ref{equitable partition}). We have \eqref{Strongly regular Neumaier graph Equation 2} by evaluating the eigenvalues of $\Gamma$ in Proposition \ref{Integer eigenvalues} and \cite[Theorem~1.3.1(iii)]{BCN}.
\end{proof}
\begin{lem} (cf. \cite[Theorem~3.3]{EGP}) \label{nexus one}
Let $\Gamma$ be a Neumaier graph with parameters $(n,k,\lambda;1,c)$. Then $\lambda=c-2$ and $k-2c+3>0$.
\end{lem}
\begin{proof}
We have $\lambda=c-2$ by Lemma \ref{Double counting 1}. Moreover, the regular induced subgraph on $S \setminus C$ has valency $c-2$ by Theorem \ref{main theorem 1}(\ref{equitable partition}) since $\lambda=c-2$ and $a=1$. This implies that $k-c+1>c-2$ because the number of vertices in $S \setminus C$ is $k-c+1$. Therefore $k-2c+3>0$.
\end{proof}
\begin{prop} \label{Special cases}
Let $\Gamma$ be a Neumaier graph with parameters \, \, \, $(n,k,\lambda;a,c)$. If $c=k$, then $\Gamma$ is the cycle graph $C_{4}$.
\end{prop}
\begin{proof}
We have $k=(n-k)a$ by Lemma \ref{Counting}, and $(k-1)(k-\lambda-1)=(n-k-1)a$ by Lemma \ref{Double counting 2}. This implies that $(k-1)(k-\lambda-1)=k-a$. It follows that $a=1$ since $k-1 \geq k-a$. Therefore $k<3$ by Lemma \ref{nexus one}. This implies that $n=4$, $k=2$, and therefore $\Gamma$ is the cycle graph $C_{4}$.
\end{proof}
\begin{prop} \label{Neumaier circulant}
There is no nontrivial strongly regular Neumaier circulant graph.
\end{prop}
\begin{proof}
The Paley graph with a prime number of vertices is the only nontrivial strongly regular circulant graph (cf. \cite[Corollary~2.10.6]{BCN}). This completes the proof by Proposition \ref{Integer eigenvalues} because it has non-integer eigenvalues.
\end{proof}
\begin{rem}
A conference graph is a strongly regular graph with parameters $(4\mu+1,2\mu,\mu-1,\mu)$. A conference graph has exactly three distinct eigenvalues $\{ 2\mu, \frac{-1\pm\sqrt{1+4\mu}}{2}\}$ (cf. \cite[Theorem~1.3.1(ii)]{BCN}). The Paley graph is an example of a conference graph (see \cite[\S~1.3]{BCN}).
\end{rem}
\section{\bf Vertex-transitive Neumaier graphs with small valency}
In this section, we study vertex-transitive Neumaier graphs with small valency at most $10$. Recall that a vertex-transitive Neumaier graph has diameter $2$ (cf. Theorem \ref{main theorem 1} \eqref{vertex-transitive Neumaier graph diameter}).
\begin{lem} \label{Number of vertices}
Let $\Gamma$ be a Neumaier graph with parameters $(n,k,\lambda;a,c)$. If $k \leq 10$, then $n \leq 36$.
\end{lem}
\begin{proof}
By Lemma \ref{Counting}, we have $c(k-c+1) \geq n-c$. The maximum value of $n$, which is $36$, is attained when $k=10$ and $c=6$. This completes the proof.
\end{proof}
\begin{exam} \label{Complete multipartite}
The complete multipartite graph $K_{\underbrace{m,m,\ldots,m}_{n \, times}}$ is a strongly regular Neumaier graph with parameters $(nm,(n-1)m,(n-2)m,(n-1)m;n-1,n)$. It is also a Cayley graph over the group $G=\mathbb{Z}_{m} \times \mathbb{Z}_{n}$ with the connection set $S=G \setminus H$, where $H$ is a subgroup of the group $G$ of order $m$ (cf. \cite[Prop.~2.6]{AJ}).
\end{exam}
\begin{prop} \label{subgroup regular clique}
Let $Cay(G,S)$ be a Cayley graph. Assume that $C \subset G$ is a regular clique with nexus $a$. If $C$ is a subgroup of $G$, then
\begin{equation*}
|S|=([G:C]-1)a+|C|-1.
\end{equation*}
\end{prop}
\begin{proof}
We count the number of vertices adjacent to $e$, which equals $|S|$. Since $C$ is a clique, the number of adjacent vertices to $e$ in $C$ is $|C|-1$. Note that every right coset of $C$ in $G$ is a regular clique with nexus $a$. Therefore, excluding $C$, $e$ has exactly $a$ adjacent vertices in each of the remaining right cosets of $C$. This completes the proof.
\end{proof}
\begin{exam} \label{Lattice}
The lattice graph $L_{2}(n,n)$ is the line graph of the complete bipartite graph $K_{n,n}$. It is a strongly regular Neumaier graph with parameters $(n^{2},2(n-1),n-2,2;1,n)$. It is also a Cayley graph over the group $G=\mathbb{Z}_{n} \times \mathbb{Z}_{n}$ with the connection set $$S=\{(0,1),(0,2),\ldots,(0,n-1),(1,0),(2,0),\ldots,(n-1,0)\}$$ (cf. \cite[Thm.~4.7]{ADJ}). We note that $$\{(0,0),(0,1),(0,2),\ldots,(0,n-1)\},$$ and $$\{(0,0),(1,0),(2,0),\ldots,(n-1,0)\}$$ are two subgroups of $G$ that are regular cliques with nexus $1$ (see Proposition \ref{subgroup regular clique}). The complement of $L_{2}(n,n)$ is also a strongly regular Neumaier Cayley graph with parameters $(n^{2},(n-1)^{2},(n-2)^{2},(n-1)(n-2);n-2,n)$.
\end{exam}
\begin{exam} \label{Schlafli}
The Schl\"{a}fli graph is the unique strongly regular graph with parameters $(27,16,10,8)$ (cf. \cite[\S~10.10]{BV}). It can be constructed as a Cayley graph and therefore its complement is also a Cayley graph (cf. \cite[\S~4.1]{ADJ}). We checked with \verb"GAP" \cite{GAP}, using \verb"AGT" package \cite{agt}, that the Schl\"{a}fli graph is not a Neumaier graph but its complement graph is a strongly regular Neumaier graph with parameters $(27,10,1,5;1,3)$.
\end{exam}
\begin{exam} \label{Clebsch}
The Clebsch graph is the unique strongly regular graph with parameters $(16,10,6,6)$ (cf. \cite[\S~10.7]{BV}). It can be constructed as a Cayley graph (cf. \cite[\S~4.1]{ADJ}). We checked with \verb"GAP" \cite{GAP}, using \verb"AGT" package \cite{agt}, that the Clebsch graph and its complement are not Neumaier graphs.
\end{exam}
\begin{exam} \label{Triangular}
The triangular graph $T(n)$ is the line graph of the complete graph $K_{n}$, where $n\geq 5$. It is a strongly regular Neumaier graph with parameters $(\binom{n}{2},2(n-2),n-2,4;2,n-1)$ (cf. \cite[Example~3]{EGP}). The triangular graph is a Cayley graph if and only if $n \equiv 3$ (mod $4$) and $n$ is a prime power (see \cite[\S~4.2]{ADJ}). The symmetric group $S_{n}$ is the automorphism group of $T(n)$. Therefore the triangular graph and its complement are vertex-transitive (see \cite[\S~11.3.5]{BV}).
\end{exam}
\begin{lem} \label{The complement of triangular graph}
The complement of $T(n)$ is a Neumaier graph if and only if $n$ is even. In this case, if the Neumaier graph $\overline{T(n)}$ has parameters $(\binom{n}{2},k,\lambda,\mu;a,c)$, then $n=2a+4$ and $c=a+2$.
\end{lem}
\begin{proof}
Let $\overline{T(n)}$ be a Neumaier graph with parameters $(\binom{n}{2},k,\lambda,\mu;a,c)$. Then $-\frac{\mu}{a}$ is the negative eigenvalue of $\overline{T(n)}$ by Proposition \ref{Integer eigenvalues}. This implies that $-\frac{\mu}{a}=3-n$, where $\mu=\frac{n(n-1)-6(n-2)}{2}$ (cf. \cite[\S~1.1.2, \S~1.1.7]{BV}). It follows that $2a(n-3)=n(n-1)-6(n-2)$ and therefore $n=2a+4$. This implies that $c=a+2$ that is the number of mutually disjoint edges in the complete graph $K_{n}$. The other way around by noting that the set of $a+2$ mutually disjoint edges in the complete graph $K_{n}$, where $n=2a+4$, is a regular clique with nexus $a$ in $\overline{T(n)}$. This completes the proof.
\end{proof}
\begin{exam}\label{srg16}
There are two strongly regular graphs with parameters $(16,6,2,2)$; the lattice graph $L_{2}(4,4)$ and the Shrikhande graph (cf. \cite[\S~10.6]{BV}). We checked with \verb"GAP" \cite{GAP}, using \verb"AGT" package \cite{agt}, that the Shrikhande graph is not a Neumaier graph but its complement graph is a Neumaier graph with parameters $(16,9,4,6;2,4)$. We note that the Shrikhande graph is the Cayley graph
\begin{equation*}
Cay(\mathbb{Z}_{4} \times \mathbb{Z}_{4},\{\pm(1,0),\pm(0,1),\pm(1,1)\}),
\end{equation*}
and therefore its complement is also a Cayley graph (cf. \cite[\S~2.1]{DJ}).
\end{exam}
\begin{prop} \label{Strongly regular Neumaier Cayley graph up to valency 10}
Let $\Gamma$ be a strongly regular graph with valency at most $10$. Then $\Gamma$ is a (vertex-transitive) Neumaier graph if and only if $\Gamma$ is isomorphic to one of the following graphs:
\begin{itemize}
\item The complete multipartite graph;
\item The lattice graph;
\item The complement of the lattice graph;
\item The triangular graph;
\item The complement of the triangular graph $T(6)$;
\item The complement of the Shrikhande graph;
\item The complement of the Schl\"{a}fli graph.
\end{itemize}
\end{prop}
\begin{proof}
The feasible parameters for strongly regular graphs with at most $512$ vertices are available in \cite[Table~12.1]{BV}. Therefore it is sufficient to investigate which one of the parameters can be putative for a Neumaier graph with at most $36$ vertices (cf. Lemma \ref{Number of vertices}). Let $\Gamma$ be a strongly regular Neumaier graph with parameters $(n,k,\lambda,\mu;a,c)$. Using Lemma \ref{Counting} and Proposition \ref{Integer eigenvalues} for $2<c<k$, we checked the putative parameters for $\Gamma$. The parameters for $\Gamma$ are listed in Table \ref{Table 1}. Note that the complete multipartite graphs are not in the table (see Example \ref{Complete multipartite} and Lemma \ref{Complete multipartite graph}). Note also that all the listed graphs are vertex-transitive. This completes the proof.
\end{proof}
\begin{center}
\begin{table}[h]
\begin{adjustbox}{width=1\textwidth}
\begin{tabular}{|l|l|c|c|c|c|c|}
  \hline
  Parameters & Name & Neumaier & Cayley & Vertex-transitive & Reference \\
  \hline
  $(9,4,1,2;1,3)$ & $L_{2}(3,3)$ &  Yes & Yes & Yes & Ex. \ref{Lattice} \\
  $(10,6,3,4;2,4)$ & $T(5)$ & Yes & No & Yes & Ex. \ref{Triangular}\\
  $(15,8,4,4;2,5)$ & $T(6)$ & Yes & No & Yes & Ex. \ref{Triangular}\\
  $(15,6,1,3;1,3)$ & $\overline{T(6)}$ & Yes & No & Yes & Lem. \ref{The complement of triangular graph}\\
  $(16,6,2,2;1,4)$ & $L_{2}(4,4)$ & Yes & Yes & Yes & Ex. \ref{Lattice}\\
   & $Shrikhande$ & No & Yes & Yes & Ex. \ref{srg16} \\
  $(16,9,4,6;2,4)$ & $\overline{L_{2}(4,4)}$ & Yes & Yes & Yes & Ex. \ref{Lattice}\\
   & Complement of Shrikhande & Yes & Yes & Yes & Ex. \ref{srg16}\\
  $(16,10,6,6;3,6)$ & Clebsch & No & Yes & Yes & Ex. \ref{Clebsch}\\
  $(21,10,5,4;2,6)$ & $T(7)$ & Yes & Yes & Yes & Ex. \ref{Triangular}\\
  $(25,8,3,2;1,5)$ & $L_{2}(5,5)$ & Yes & Yes & Yes & Ex. \ref{Lattice}\\
  $(27,10,1,5;1,3)$ & Complement of Schl\"{a}fli & Yes & Yes & Yes & Ex. \ref{Schlafli}\\
  $(36,10,4,2;1,6)$ & $L_{2}(6,6)$ & Yes & Yes & Yes & Ex. \ref{Lattice}\\
  \hline
\end{tabular}
\end{adjustbox}
\caption{Strongly regular Neumaier graphs with small valencies at most $10$} \label{Table 1}
\end{table}
\end{center}
The feasible parameters for a vertex-transitive strictly Neumaier graph with at most $47$ vertices are available in \cite[\S~4.4.1]{Evans}. There is only one strictly Neumaier graph with parameters $(16,9,4;2,4)$ (cf. \cite[Theorem~8]{ADZ}). This graph can be constructed as Cayley graph over the abelian group $\mathbb{Z}_{8} \times \mathbb{Z}_{2}$ (cf. \cite[\S~6]{EGP}). It can also be constructed as Cayley graph over the dihedral group $D_{16}$ as follows.
\begin{lem} \label{Parameters 16}
The unique strictly Neumaier graph with parameters $(16,9,4;2,4)$ is a Cayley graph over the dihedral group $D_{16}$.
\end{lem}
\begin{proof}
This graph is a Cayley graph $Cay(G,S)$, where $G=D_{16}=\seq{a,b|a^{8}=b^{2}=(ba)^{2}=e}$ and $S=\{a,a^{-1},a^{2},a^{-2},b,ba,ba^{3},ba^{4},ba^{6}\}$.
\end{proof}
There are only four vertex-transitive strictly Neumaier graphs with parameters $(24,8,2;1,4)$ (cf. \cite[\S~4.4.1]{Evans}).
\begin{lem} \label{Parameters 24}
The four vertex-transitive strictly Neumaier graphs with parameters $(24,8,2;1,4)$ are Cayley graphs.
\end{lem}
\begin{proof}
These four vertex-transitive strictly Neumaier graphs are available in \verb"GRAPE" format of \verb"GAP" (cf. \cite[\S~4.A(2)]{Evans}). We checked them with \verb"GAP" \cite{GAP} that all of them are Cayley graphs as follows, respectively.
\begin{itemize}
\item $G=S_{4}$, \\ $S=\{(1,3)(2,4), (1,4)(2,3), (1,2,4), (1,4,2), (1,3,4) , \\ (1,4,3), (1,2,4,3), (1,3,4,2)\}$;
\item $G=A_{4} \times \mathbb{Z}_{2}=A_{4} \times \seq{(5,6)}$, \\ $S=\{(1,3)(2,4), (1,2)(3,4), (1,2,4), (1,4,2), (1,2,3), \\ (1,3,2) , (1,3)(2,4)(5,6), (1,4)(2,3)(5,6)\}$;
\item $G=S_{4}$, \\ $S=\{(1,3)(2,4), (1,4)(2,3), (1,2,4), (1,4,2), (1,3,4) , \\ (1,4,3), (1,4), (2,3)\}$;
\item $G=A_{4} \times \mathbb{Z}_{2}=A_{4} \times \seq{(5,6)}$, \\ $S=\{ (1,3)(2,4), (1,4)(2,3), (1,2,4), (1,4,2), (1,3,4) , \\ (1,4,3), (1,4)(2,3)(5,6),(5,6)\}$.
\end{itemize}
\end{proof}
There are only two non-isomorphic vertex-transitive Neumaier graphs with parameters $(28,9,2;1,4)$ (cf. \cite[\S~4.4.1]{Evans}).
\begin{lem} \label{Parameters 28}
The two vertex-transitive strictly Neumaier graphs with parameters $(28,9,2;1,4)$ are Cayley graphs.
\end{lem}
\begin{proof}
These two vertex-transitive strictly Neumaier graphs are available in \verb"GRAPE" format of \verb"GAP" (cf. \cite[\S~4.A(3)]{Evans}). We checked them with \verb"GAP" \cite{GAP} that all of them are Cayley graphs as follows, respectively.
\begin{itemize}
\item $G=\mathbb{Z}_{28}$, $S=\{\pm 1,\pm 4,\pm 5, \pm 7, 14\}$;
\item $G=\mathbb{Z}_{2} \times \mathbb{Z}_{14}$, $S=\{(1,0),(0,\pm 1),(0,7),(1,\pm 2), (1,\pm 3), (1,7)\}$.
\end{itemize}
\end{proof}
\begin{prop} \label{strictly neumaier cayley graphs}
Let $\Gamma$ be a strictly Neumaier Cayley graph with with parameters $(n,k,\lambda;a,c)$, where $k \leq 10$. Then
\begin{equation*}
(n,k,\lambda;a,c) \in \{(16,9,4;2,4),(24,8,2;1,4),(28,9,2;1,4)\}.
\end{equation*}
\end{prop}
\begin{proof}
Using Lemma \ref{Number of vertices}, the feasible parameters for $\Gamma$ are available in \cite[\S~4.4.1]{Evans}. This completes the proof by Lemmas \ref{Parameters 16}, \ref{Parameters 24}, and \ref{Parameters 28}.
\end{proof}
\section{\bf Neumaier graphs with diameter $3$}
There are not many known examples of Neumaier graphs with diameter $3$. One of them appeared in Evans' thesis; this is a Neumaier graph with parameters $(24,8,2;1,4)$ (cf. \cite[\S~4.4.2]{Evans}). Recall that a vertex-transitive Neumaier graph has diameter $2$ and therefore there is no vertex-transitive Neumaier graph with diameter $3$ (see Theorem \ref{main theorem 1} \eqref{vertex-transitive Neumaier graph diameter}).

\vskip 0.4 true cm
\subsection*{\textbf{Acknowledgments}}
\noindent  Mojtaba Jazaeri would like to thank the anonymous referee for the invaluable comments and suggestions on this paper. Mojtaba Jazaeri would like to thank Paul Terwilliger for his invaluable comments and suggestions on this paper. The author is grateful to the Research Council of Shahid Chamran University of Ahvaz for financial support (SCU.MM1403.29248).

\bigskip

\end{document}